\documentclass[12pt]{amsproc}

\usepackage{amscd, amsmath, amsthm, amssymb, mathtools, mathrsfs, stmaryrd}
\usepackage{ascmac}
\usepackage{comment}
\usepackage{bm, bbm}
\allowdisplaybreaks

\usepackage{CJKutf8}

\usepackage[top=30mm, bottom=30mm, left=25mm, right=25mm]{geometry}

\usepackage{hyperref}

\usepackage{tikz}
\usetikzlibrary{intersections, calc, arrows.meta}

\usepackage{here}
\usepackage{time}
\usepackage[abbrev]{amsrefs}

\usepackage{xcolor}
\usepackage[capitalize,nameinlink,noabbrev,nosort]{cleveref}
\hypersetup{
	colorlinks=true,       
	linkcolor=brown,          
	citecolor=brown,        
	filecolor=brown,      
	urlcolor=brown,           
}

\makeatletter
\@namedef{subjclassname@2020}{\textup{2020} Mathematics Subject Classification}
\makeatother

\newtheorem*{theorem*}{\hspace{-6.3mm}\textbf{Theorem}}  

\newtheorem{theoremcounter}{Theorem Counter}[section]

\theoremstyle{remark}

\theoremstyle{definition}
\newtheorem{definition}[theoremcounter]{Definition}
\newtheorem{example}{Example}

\theoremstyle{plain}
\newtheorem{lemma}[theoremcounter]{Lemma}
\newtheorem{proposition}[theoremcounter]{Proposition}
\newtheorem{corollary}[theoremcounter]{Corollary}

\newtheorem{theorem}[theoremcounter]{Theorem}

\numberwithin{equation}{section}

\newcommand{\Z}{\mathbb{Z}}
\newcommand{\Q}{\mathbb{Q}}
\newcommand{\R}{\mathbb{R}}
\newcommand{\C}{\mathbb{C}}

\newcommand{\dd}{\mathrm{d}}
\newcommand{\bbH}{\mathbb{H}}

\DeclareMathOperator{\ImNew}{Im}
\renewcommand{\Im}{\ImNew}
\DeclareMathOperator{\ReNew}{Re}
\renewcommand{\Re}{\ReNew}
\DeclareMathOperator{\divNew}{div}
\renewcommand{\div}{\divNew}

\DeclareMathOperator{\SL}{SL}
\DeclareMathOperator{\PSL}{PSL}

\DeclareMathOperator{\sgn}{sgn}

\newcommand{\pmat}[1]{\begin{pmatrix}#1\end{pmatrix}}

\newcommand{\smat}[1]{\bigl(\begin{smallmatrix}#1\end{smallmatrix}\bigr)}

\begin{document}

\title[]{The Fourier coefficients and singular moduli of the elliptic modular function $j(\tau)$, revisited}

\author[]{Toshiki Matsusaka}
\address{Faculty of Mathematics, Kyushu University, Motooka 744, Nishi-ku, Fukuoka 819-0395, Japan}
\email{matsusaka@math.kyushu-u.ac.jp}





\maketitle

\begin{center}
Dedicated to Professor Masanobu Kaneko on his 60+4th birthday
\end{center}

\begin{abstract}
	Kaneko's formula expresses the Fourier coefficients of the elliptic modular $j$-function as finite sums of singular moduli. First published as a short article in 1996, it was presented as a consequence of Zagier's work inspired by Borcherds products. Since then, the formula has developed into a broader framework that links the Fourier coefficients of modular forms to the special values of modular functions, extending in various directions. This article surveys these subsequent developments.
\end{abstract}


\section{Introduction}\label{sec:Intro}

The elliptic modular $j$-function, which Kaneko~\cite{Kaneko2001} once described as \mbox{``\begin{CJK}{UTF8}{ipxm}愛惜措く能わざる\end{CJK}"} (that is, an object too profoundly cherished to part with), is defined as a modular function for $\SL_2(\Z)$ by the following expression:
\[
	j(\tau) \coloneqq \frac{E_4(\tau)^3}{\Delta(\tau)} = q^{-1} + 744 + 196884q + 21493760q^2 + \cdots,
\]
where $E_4(\tau)$ is the Eisenstein series of weight $4$, $\Delta(\tau)$ is the discriminant cusp form of weight $12$, and we set $q \coloneqq e^{2\pi i \tau}$. Among its many remarkable features, one of particular interest in this article is the special values of $j(\tau)$ at imaginary quadratic irrationalities in the upper-half plane $\bbH \coloneqq \{\tau \in \C : \Im(\tau) > 0\}$. The values, known as \emph{singular moduli}, exhibit striking arithmetic properties. Well-known examples include
\[
	j \left(\frac{1+\sqrt{-3}}{2}\right) = 0, \quad j(i) = 1728, \quad \text{and} \quad j \left(\frac{1+\sqrt{-15}}{2}\right) = \frac{-191025 - 85995\sqrt{5}}{2}.
\]
Kaneko's work~\cite{Kaneko1996} provided deep insights into the nature of these special values.

In brief, Kaneko's theorem states that the Fourier coefficients of the $j$-function can be expressed as finite sums of singular moduli. To present this result precisely, we first recall some basic notations related to binary quadratic forms. Let $-d < 0$ be a discriminant satisfying $-d \equiv 0, 1 \pmod{4}$. We denote by $\mathcal{Q}_d$ the set of integral positive definite binary quadratic forms $Q(x, y) = [a, b, c] = ax^2 + bxy + cy^2$ with discriminant $-d = b^2 - 4ac$. The group $\Gamma = \PSL_2(\Z)$ acts on $\mathcal{Q}_d$ via
\[
	\left(Q \circ \pmat{A & B \\ C & D} \right)(x, y) \coloneqq Q(Ax+By, Cx+Dy).
\]
It is a classical fact that this action has finitely many orbits. The number of orbits $|\mathcal{Q}_d/\Gamma|$ is called the \emph{class number} of discriminant $-d$. 
For each $Q \in \mathcal{Q}_d$, let $\Gamma_Q$ denote the stabilizer of $Q$ in $\Gamma$. Its order is $3, 2$, or $1$ according to whether $Q$ is $\Gamma$-equivalent to a form of the shape $a(x^2 + xy + y^2)$, $a(x^2 + y^2)$, or otherwise, respectively. We associate to each form $Q \in \mathcal{Q}_d$ the unique room $\alpha_Q \in \bbH$ of the equation $Q(\tau, 1) = 0$. With this setup, Zagier~\cite{Zagier2002} studied the following quantity, called the \emph{traces of singular moduli}. For $m \ge 1$, it is defined by
\begin{align}\label{eq:trace-singular}
	\mathbf{t}_m(d) \coloneqq \sum_{Q \in \mathcal{Q}_d/\Gamma} \frac{1}{|\Gamma_Q|} j_m(\alpha_Q),
\end{align}
where $j_m(\tau)$ is a polynomial in $j(\tau)$ whose Fourier expansion begins as $j_m(\tau) = q^{-m} + O(q)$. The first few examples of $j_m(\tau)$ are given by
\begin{align*}
	j_1(\tau) &= j(\tau) - 744 = q^{-1} + 196884q + \cdots,\\
	j_2(\tau) &= j(\tau)^2 - 1488j(\tau) + 159768 = q^{-2}+42987520 q + \cdots.
\end{align*}

Before stating the formula of interest, we first examine several concrete values of $\mathbf{t}_2(d)$ to see how it works. Since $\mathcal{Q}_3/\Gamma = \{[x^2 + xy + y^2]\}$ and $\mathcal{Q}_4/\Gamma = \{[x^2 + y^2]\}$, we have
\begin{align*}
	\mathbf{t}_2(3) &= \frac{1}{3}((j(e^{2\pi i/3})^2 -1488j(e^{2\pi i/3}) + 159768) = 53256,\\
	\mathbf{t}_2(4) &= \frac{1}{2}(j(i)^2 -1488j(i) + 159768) = 287244.
\end{align*}
Remarkably, these values satisfy the relation
\[
	\mathbf{t}_2(4) + 2 \mathbf{t}_2(3) + 12 = 2 \times 196884,
\]
which yields twice the first Fourier coefficient of $j(\tau)$. More generally, Kaneko established the following identity.

\begin{theorem}[Kaneko~\cite{Kaneko1996}]\label{thm:Kaneko-formula}
	For $m=2$, we define $\mathbf{t}_2(d)$ as above for negative discriminants $-d < 0$. In addition, we set $\mathbf{t}_2(0) = 6, \mathbf{t}_2(-1) = -1, \mathbf{t}_2(-4) = -2$, and define $\mathbf{t}_2(d) = 0$ for all other integers $d$. Then for any integer $n \ge 1$, the following identity holds:
	\[
		\sum_{r \in \Z} \mathbf{t}_2(4n-r^2) = 2n \mathrm{Coeff}_{q^n} (j),
	\]
	where $\mathrm{Coeff}_{q^n} (j)$ is the $n$-th Fourier coefficient of $j(\tau)$.
\end{theorem}

It is immediately evident that Kaneko's identity closely parallels the following classical formula due to Hurwitz. Define $j_0(\tau) \coloneqq 1$, and let $\mathbf{t}_0(d)$ be defined via \eqref{eq:trace-singular}, applied to $j_0$. In addition, we set $\mathbf{t}_0(0) = -1/12$ and $\mathbf{t}_0(d)$ for all other integers $d$. The weighted class number $\mathbf{t}_0(d)$, also denoted by $H(d)$, is known as the \emph{Kronecker--Hurwitz class number}.

\begin{theorem}[Hurwitz~\cite{Hurwitz1885}]\label{thm:Hurwitz}
	For any integer $n \ge 1$, we have
	\[
		\sum_{r \in \Z} \mathbf{t}_0(4n-r^2) = \sum_{d \mid n} \max \left(d, \frac{n}{d}\right).
	\]
\end{theorem}

Behind this similarity lies a deeper structure that has gradually come to light through subsequent research.
This article surveys developments stemming from Kaneko's formula, with particular emphasis on those concerning the connection between Fourier coefficients of modular forms and the special values of modular functions. When our approach differs from or simplifies the original, we provide full proofs. Otherwise, we state the results and offer concise summaries of the underlying ideas.

\section*{Acknowledgement}

This survey is based on a talk given at the conference titled ``Modular Forms and Multiple Zeta Values --Conference in Honor of Masanobu Kaneko's 60+4th Birthday--", held at Kindai University in February 2025. I would like to express my sincere gratitude to Professor Masanobu Kaneko, who was my advisor during my student years, for his invaluable guidance and support throughout my career, even after graduation. I am also grateful to the organizers and all those involved in the conference for their efforts. This work was supported by JSPS KAKENHI (JP21K18141 and JP24K16901) and the MEXT Initiative through Kyushu University's Diversity and Super Global Training Program for Female and Young Faculty (SENTAN-Q).

\section{Preliminaries on Jacobi forms}

The proof of Kaneko's formula and its subsequent developments are fundamentally based on the modularity satisfies by traces of singular moduli, as established by Zagier~\cite{Zagier2002} and by Bruinier--Funke~\cite{BruinierFunke2006}. In this section, we review the necessary background on Jacobi forms, which play a fundamental role in describing this modularity.

\subsection{Jacobi forms}

We begin by recalling the notion of Jacobi forms. For further details, see Eichler--Zagier~\cite{EichlerZagier1985}. Let $k$ and $m$ be positive integers. For a function $\phi: \bbH \times \C \to \C$, we define the actions of $\gamma = \smat{a & b \\ c & d} \in \SL_2(\Z)$ and $(\lambda, \mu) \in \Z^2$ by
\begin{align*}
	(\phi|_{k,m} \gamma)(\tau, z) &\coloneqq (c\tau+d)^{-k} e^{-2\pi im \frac{cz^2}{c\tau+d}} \phi \left(\frac{a\tau+b}{c\tau+d}, \frac{z}{c\tau+d}\right),\\
	(\phi|_m (\lambda, \mu)) (\tau, z) &\coloneqq e^{2\pi im(\lambda^2 \tau+2\lambda z)} \phi(\tau, z + \lambda \tau + \mu).
\end{align*}

\begin{definition}
	A holomorphic function $\phi: \bbH \times \C \to \C$ is called a \emph{weakly holomorphic Jacobi form of weight $k$ and index $m$} if it satisfies the following:
	\begin{itemize}
		\item $\phi|_{k,m} \gamma = \phi$ for any $\gamma \in \SL_2(\Z)$.
		\item $\phi|_m (\lambda, \mu) = \phi$ for any $(\lambda, \mu) \in \Z^2$.
		\item It has the Fourier series expansion of the form
		\[
			\phi(\tau, z) = \sum_{\substack{n, r \in \Z \\ 4mn-r^2 \gg -\infty}} c(n, r) q^n \zeta^r, \quad (q = e^{2\pi i\tau}, \zeta = e^{2\pi iz}),
		\]
		where the condition $4mn - r^2 \gg -\infty$ means that there exists $N_0 \in \Z$ such that $c(n,r) = 0$ for all $4mn-r^2 \le N_0$.
	\end{itemize}
	We denotes by $J_{k,m}^!$ the $\C$-vector space consisting of such Jacobi forms.
\end{definition}

As shown in \cite[Theorem 2.2]{EichlerZagier1985}, the coefficients $c(n,r)$ depend only on $4mn-r^2$ and $r \pmod{2m}$. Therefore, they can be written in the form
\[
	c(n,r) = c_r(4mn-r^2), \qquad c_{r'}(N) = c_r(N) \text{ for } r' \equiv r \pmod{2m}.
\]
Accordingly, we assume from now on that the Jacobi form $\phi$ has a Fourier series expansion of the form
\begin{align}\label{eq:Jacobi-Fourier}
	\phi(\tau, z) = \sum_{n \gg -\infty} \sum_{r \in \Z} c_r(4mn-r^2) q^n \zeta^r.
\end{align}
The condition $4mn-r^2 \gg -\infty$ is assumed throughout, though we omit it from the notation below. For $\mu \in \Z/2m\Z$ and $N \in \Z$ satisfying $N \equiv -\mu^2 \pmod{4m}$, we set
\[
	c_\mu(N) \coloneqq c \left(\frac{N+r^2}{4m}, r\right)
\]
for any $r \in \Z$ with $r \equiv \mu \pmod{2m}$. In this setting, for each $\mu \in \Z/2m\Z$, we define
\begin{align*}
	h_\mu(\tau) &\coloneqq \sum_{\substack{N \gg -\infty \\ N \equiv -\mu^2\ (4m)}} c_\mu(N) q^{N/4m},\\
	\theta_{m, \mu}(\tau, z) &\coloneqq \sum_{\substack{r \in \Z \\ r \equiv \mu\ (2m)}} q^{r^2/4m} \zeta^r.
\end{align*}
Then the following decomposition, known as the \emph{theta decomposition} (see, \cite[Section 5]{EichlerZagier1985}), holds:
\begin{align}\label{eq:theta-decomp}
	\phi(\tau, z) = \sum_{\mu \in \Z/2m\Z} h_\mu(\tau) \theta_{m,\mu}(\tau, z).
\end{align}
The components $h_\mu(\tau)$ of the theta decomposition satisfy the following transformation laws, (see~\cite[Section 5]{EichlerZagier1985}):
\begin{align}\label{eq:vect-modular}
	h_\mu(\tau+1) = \mathbf{e}\left(-\frac{\mu^2}{4m}\right) h_\mu(\tau), \qquad h_\mu \left(-\frac{1}{\tau}\right) = \frac{\tau^k}{\sqrt{2m\tau/i}} \sum_{\nu \in \Z/2m\Z} \mathbf{e}\left(\frac{\mu \nu}{2m}\right) h_\nu(\tau).
\end{align}
Here we set $\mathbf{e}(x) \coloneqq e^{2\pi i x}$. Conversely, any collection of functions $(h_\mu(\tau))_{\mu \in \Z/2m\Z}$ satisfying the above transformation laws defines a Jacobi form of weight $k$ and index $m$.

\subsection{Connections with modular forms}

In the following, we present two approaches to constructing modular forms from Jacobi forms. We denote by $M_k^!(\Gamma)$ the $\C$-vector space of weakly holomorphic modular forms of weight $k$ on the group $\Gamma$, that is, functions that are holomorphic on $\bbH$, invariant under the usual slash action $|_k \gamma$ for $\gamma \in \Gamma$, and meromorphic at the cusps, (i.e., allowed to have poles only at the cusps of $\Gamma$).

\begin{lemma}[{\cite[Section 3]{EichlerZagier1985}}]
	For positive integers $k$ and $m$, we define the weighted heat operator $L_{k,m}$ by
	\[
		L_{k,m} \coloneqq -D - \frac{1}{16\pi^2 m} \left(\frac{\partial^2}{\partial z^2} + \frac{2k-1}{z} \frac{\partial}{\partial z} \right),
	\]
	where $D \coloneqq \frac{1}{2\pi i} \frac{\dd}{\dd \tau}$. Then, for any $\gamma \in \SL_2(\Z)$, we have
	\[
		L_{k,m}(\phi|_{k,m} \gamma) = (L_{k,m} \phi)|_{k+2, m} \gamma.
	\]
\end{lemma}

The following construction follows immediately from this lemma.

\begin{proposition}\label{prop:RC-bracket}
	Let $\nu \ge 0$ be an integer. For any $\phi \in J_{k,m}^!$, we have
	\[
		F_{\phi, \nu}(\tau) \coloneqq \lim_{z \to 0} L_{k+2\nu-2, m} \circ \cdots \circ L_{k+2, m} \circ L_{k, m} \phi(\tau, z) \in M_{k+2\nu}^!(\SL_2(\Z)).
	\]
\end{proposition}

\begin{example}
	When $\nu = 0$, we have
	\[
		F_{\phi, 0}(\tau) = \phi(\tau, 0) = \sum_{n \gg -\infty} \left(\sum_{r \in \Z} c_r(4mn-r^2) \right) q^n \in M_k^!(\SL_2(\Z)).
	\]
	In particular, if $\sum_{r \in \Z} c_r(4mn-r^2) = 0$ for all $n < 0$, we have $\phi(\tau, 0) \in M_k(\SL_2(\Z))$.
\end{example}

In general, the Fourier expansion of $F_{\phi, \nu}(\tau)$ is given by the following.

\begin{proposition}\label{prop:RC-expansion}
	For $k \ge 2$, we define the Gegenbauer polynomials $p_{k,l}(r,n)$ by
	\[
		\sum_{l=0}^\infty \binom{l/2+k-2}{k-2}p_{k,l}(r,n) X^l = \frac{1}{(1-rX+nX^2)^{k-1}},
	\]
	or equivalently,
	\begin{align*}
		a_{k,l}(r,n) \coloneqq \binom{l/2+k-2}{k-2} p_{k,l}(r,n) = \sum_{0 \le j \le l/2} (-1)^j \binom{l+k-2-j}{l-2j} \binom{j+k-2}{k-2} n^j r^{l-2j}.
	\end{align*}
	For $\phi \in J_{k,m}^!$, we have
	\[
		F_{\phi, \nu}(\tau) = \sum_{n \gg -\infty} \left(\sum_{r \in \Z} p_{k,2\nu} \left(\frac{r}{\sqrt{m}}, n \right) c_r(4mn-r^2) \right) q^n.
	\]
\end{proposition}

\begin{proof}
	This statement is given in~\cite[Theorem 3.1]{EichlerZagier1985}. Here, we provide a sketch of the proof using a slightly different approach. The equivalence of the two definitions of $p_{k,l}(r,n)$ can be shown verifying that they satisfy the same recurrence relation. For $0 \le j \le l/2$ and $k \ge 2$, we define
	\[
		S_{j,k,l}(r,n) \coloneqq (-1)^j \binom{l+k-2-j}{l-2j} \binom{j+k-2}{k-2} n^j r^{l-2j}.
	\]
	We set $S_{j,k,l}(r,n) = 0$ whenever the indices do not satisfy the conditions stated above. Then, for $l \ge 1$, it is straightforward to verify that
	\[
		S_{j,k,l}(r,n) - rS_{j,k,l-1}(r,n) + nS_{j-1,k,l-2}(r,n) = S_{j,k-1,l}(r,n).
	\]
	By summing both sides over $0 \le j \le l/2$, we obtain
	\[
		a_{k,l}(r,n) - r a_{k,l-1}(r,n) + n a_{k,l-2}(r,n) = a_{k-1,l}(r,n),
	\]
	which corresponds to the recurrence relation satisfied by the coefficients of $1/(1-rX+nX^2)^{k-1}$. Thus, the claim follows by induction on $k \ge 2$ and $l \ge 0$.
	
	For the later result, we consider a more general family of polynomials. For $h \ge 0$, we define
	\[
		p_{k,l,h}(r,n) \coloneqq \sum_{0 \le j \le l/2} (-1)^j \binom{l+k-2-j+2h}{l-2j+2h} \frac{\binom{j+k-2}{k-2}}{\binom{l/2+k-2}{k-2}} \frac{\binom{2h}{h} \binom{l/2-j+h}{h}}{\binom{l/2+k-2+h}{h} \binom{l+k-2-j+2h}{h}} n^j r^{l-2j}
	\]
	Then we can verify the Taylor expansion
	\[
		L_{k+2\nu-2, m} \circ \cdots \circ L_{k,m} q^n (\zeta^r + \zeta^{-r}) = 2 \sum_{h=0}^\infty p_{k,2\nu,h} \left(\frac{r}{\sqrt{m}}, n \right) q^n \frac{(2\pi irz)^{2h}}{(2h)!}
	\]
	by induction on $\nu$.	Indeed, as in the previous argument, we can check that the Taylor coefficients satisfy the same recurrence relation as 
	\[
		p_{k, 2\nu+2,h} \left(\frac{r}{\sqrt{m}}, n \right) = -n p_{k, 2\nu,h} \left(\frac{r}{\sqrt{m}}, n \right) + \frac{r^2}{4m} \left(1 + \frac{4\nu+2k-1}{2h+1}\right) p_{k,2\nu,h+1}\left(\frac{r}{\sqrt{m}}, n \right).
	\]
	Since $p_{k,l,0}(r,n) = p_{k,l}(r,n)$, we obtain
	\begin{align*}
		F_{\phi, \nu}(\tau) &= \frac{1}{2} \lim_{z \to 0} L_{k+2\nu-2, m} \circ \cdots \circ L_{k,m} \sum_{n \gg -\infty} \sum_{r \in \Z} c_r(4mn-r^2) q^n(\zeta^r + \zeta^{-r})\\
			&= \sum_{n \gg -\infty} \sum_{r \in \Z} c_r(4mn-r^2) p_{k, 2\nu} \left(\frac{r}{\sqrt{m}}, n \right) q^n
	\end{align*}
	as desired.
\end{proof}

\begin{example}
	The first few examples of $p_{k,2\nu}(r,n)$ are given by
	\[
		p_{k,0}(r,n) = 1, \quad p_{k,2}(r,n) = \frac{k}{2} r^2-n, \quad p_{k,4}(r,n) = \frac{(k+1)(k+2)}{12} r^4 - (k+1) nr^2 + n^2.
	\]
\end{example}

The above method produces only modular forms of level one. The next construction presents a way to obtain modular forms of level $m$ from Jacobi forms of index $m$.

\begin{proposition}\label{prop:Gphi}
	For any $\phi \in J_{k,m}^!$, we have
	\[
		G_\phi(\tau) \coloneqq \sum_{\ell=0}^{m-1} q^{\ell^2} \phi(m\tau, \ell \tau) \in M_k^!(\Gamma_0(m)).
	\]
\end{proposition}

\begin{proof}
	As holomorphy and the behavior at the cusps follow easily from the assumption $\phi \in J_{k,m}^!$, we focus here on verifying that $G_\phi|_k \gamma = G_\phi$ holds for any $\gamma = \smat{a & b \\ c & d} \in \Gamma_0(m)$. First, the modular and elliptic transformation laws yield
	\begin{align*}
		(G_\phi|_k \gamma)(\tau) &= (c\tau + d)^{-k} \sum_{\ell=0}^{m-1} e^{2\pi i \ell^2 \frac{a\tau+b}{c\tau+d}} \phi \left(\pmat{a & bm \\ c/m & d} \cdot m\tau, \frac{\ell(a\tau+b)}{c\tau + d}\right)\\
			&= \sum_{\ell=0}^{m-1} e^{2\pi i \ell^2 a(a\tau+b)} \phi (m\tau, \ell(a\tau+b))\\
			&= \sum_{\ell=0}^{m-1} q^{(a\ell)^2} \phi (m\tau, a\ell\tau).
	\end{align*}
	Since $a$ is coprime to $m$, multiplication by $a$ induces a bijection on $\Z/m\Z$. Therefore, the elliptic transformation law of $\phi$ implies that the expression equals $G_\phi$.
\end{proof}

Here as well, we describe the Fourier expansion explicitly.

\begin{proposition}\label{prop:Gphi-Fourier}
	For the coefficients $c_\mu(N)$ of a Jacobi form $\phi \in J_{k,m}^!$, we set
	\[
		c^*(N) \coloneqq \sum_{\substack{\mu \in \Z/2m \Z \\ N \equiv -\mu^2\ (4m)}} c_\mu(N).
	\]
	Then, we have
	\[
		G_\phi(\tau) = \sum_{n \gg -\infty} \left(\sum_{r \in \Z} c^*(4n-r^2) \right) q^n.
	\]
\end{proposition}

\begin{proof}
	By applying the theta decomposition~\eqref{eq:theta-decomp}, we obtain
	\begin{align*}
		G_\phi(\tau) 
		&= \sum_{\mu \in \Z/2m \Z} h_\mu(m\tau) \sum_{\ell=0}^{m-1} \sum_{r \equiv \mu\ (2m)} q^{(r+2\ell)^2/4} \\
		&= \sum_{j \in \{0, 1\}} \bigg(\sum_{\substack{\mu \in \Z/2m \Z \\ \mu \equiv j\ (2)}} h_\mu(m\tau) \bigg) \cdot \bigg(\sum_{r \equiv j\ (2)} q^{r^2/4} \bigg).
	\end{align*}
	For each $j \in \{0, 1\}$, we observe that
	\begin{align*}
		\sum_{\substack{\mu \in \Z/2m \Z \\ \mu \equiv j\ (2)}} h_\mu(m\tau) = \sum_{N \equiv -j\ (4)} \bigg(\sum_{\substack{\mu \in \Z/2m\Z \\ N \equiv -\mu^2\ (4m)}} c_\mu(N) \bigg) q^{N/4}.
	\end{align*}
	Therefore, the claim follows.
\end{proof}

\begin{example}\label{ex:m=p}
	When $m=1$, we have $G_\phi = F_{\phi, 0}$. Let $m=p$ be a prime and $k \in 2\Z$. In this case, as shown in~\cite[Theorem 2.2]{EichlerZagier1985}, the coefficients $c(n,r)$ depend only on $4pn-r^2$. Therefore, for $N \equiv 0, 3 \pmod{4}$, we have
	\begin{align*}
		c^*(N) &= \#\{\mu \in \Z/2p \Z : N \equiv -\mu^2 \pmod{4p}\} \cdot c(N)\\
			&= \left(1 + \left(\frac{-N}{p}\right) \right) c(N),
	\end{align*}
	where $\left(\frac{\cdot}{\cdot}\right)$ is the Kronecker symbol. Since $c(N) = c^*(N) = 0$ if $\left(\frac{-N}{p}\right) = -1$, we also have
	\[
		c^*(N) = 2^{1-\omega(\gcd(p,N))} c(N),
	\]
	with $\omega(n)$ denoting the number of distinct prime divisors of $n$.
\end{example}

\section{Kaneko's formula}

\subsection{Zagier's generating function}

The key to the proof of Kaneko's formula lies in a result of Zagier~\cite{Zagier2002}, which states that the generating function of $\mathbf{t}_m(d)$ is a modular form of weight $3/2$ on $\Gamma_0(4)$. Here, we give an equivalent statement in terms of Jacobi forms.

\begin{theorem}[Zagier~\cite{Zagier2002}]\label{thm:Zagier-gen}
	We put
	\[
		\mathbf{t}_m(0) = 2\sigma_1(m),
	\]
	and for negative $d$,
	\[
		\mathbf{t}_m(d) = \begin{cases}
			-\kappa &\text{if } d = -\kappa^2 \text{ for some positive integer $\kappa \mid m$},\\
			0 &\text{otherwise}.
		\end{cases}
	\]
	Then we have
	\[
		g_m(\tau, z) \coloneqq \sum_{n \gg -\infty} \sum_{r \in \Z} \mathbf{t}_m(4n-r^2) q^n \zeta^r \in J_{2,1}^!.
	\]
\end{theorem}

As an immediate consequence of this result, we obtain the following recurrence relations.

\begin{corollary}\label{cor:Kaneko-Zagier-rel}
	For any integer $n \ge 1$, we have
	\[
		\sum_{r \in \Z} \mathbf{t}_1(4n-r^2) = 0 \quad \text{ and } \quad \sum_{r \in \Z} r^2 \mathbf{t}_1(4n-r^2) = -480 \sigma_3(n).
	\]
\end{corollary}

\begin{proof}
	Applying \cref{prop:RC-bracket} and \cref{prop:RC-expansion} for $\nu = 0, 1$, we have
	\begin{align*}
		F_{g_1, 0}(\tau) &= \sum_{n =0}^\infty \left(\sum_{r \in \Z} \mathbf{t}_1(4n-r^2)\right) q^n \in M_2(\SL_2(\Z)) = \{0\},\\
		F_{g_1, 1}(\tau) &= \sum_{n=0}^\infty \left(\sum_{r \in \Z} (r^2- n) \mathbf{t}_1(4n-r^2)\right) q^n \in M_4(\SL_2(\Z)) = \C E_4.
	\end{align*}
	The first assertion follows from the first identity, and the second identity shows that
	\[
		F_{g_1, 1}(\tau) = -2 + O(q) = -2E_4(\tau) = -2 - 480 \sum_{n=1}^\infty \sigma_3(n) q^n,
	\]
	which implies the second assertion.
\end{proof}

This recurrence allows us to compute the values of $\mathbf{t}_1(d)$ inductively, even without relying on its original definition in terms of singular moduli. Conversely, Zagier proves \cref{thm:Zagier-gen} for $g_1$ by verifying that $\mathbf{t}_1(d)$ satisfies this recurrence using modular polynomials (in the context of half-integral weight modular forms). We note that it is also possible to express $g_1(\tau, z)$ in terms of half-integral weight modular forms as 
\begin{align}\label{eq:Zagier-half}
	g_1(\tau) \coloneqq -q^{-1} + 2 + \sum_{\substack{d > 0 \\ d \equiv 0, 3\ (4)}} \mathbf{t}_1(d) q^d = -\theta_1(\tau) \frac{E_4(4\tau)}{\eta(4\tau)^6},
\end{align}
where
\[
	\theta_1(\tau) \coloneqq \sum_{n \in \Z} (-1)^n q^{n^2}, \qquad \eta(\tau) \coloneqq q^{1/24} \prod_{n=1}^\infty (1-q^n).
\]
As he further pointed out, \cref{thm:Zagier-gen} and its proof can also be reinterpreted in light of Borcherds' results~\cite{Borcherds1995}. In addition, by considering the action of Hecke operators, the proof can be extended to general $m \ge 1$. The relation to the Borcherds product will be discussed in \cref{sec:Rohrlich}.

\subsection{Proof of \cref{thm:Kaneko-formula}}

The formula that Kaneko initially arrived at is not \cref{thm:Kaneko-formula}, but rather the one involving $\mathbf{t}_1(d)$ stated below.

\begin{theorem}[Kaneko~\cite{Kaneko1996}]\label{thm:Kaneko-original}
	For any integer $n \ge 1$, we have
	\[
		\mathrm{Coeff}_{q^n}(j) = \frac{1}{n} \sum_{r \in \Z} \left\{\mathbf{t}_1(n-r^2) - \frac{(-1)^{n+r}}{4} \mathbf{t}_1(4n-r^2) + \frac{(-1)^r}{4} \mathbf{t}_1(16n-r^2) \right\}.
	\]
\end{theorem}

According to Kaneko himself~\cite{Kaneko2001}, ``the theorem was first discovered experimentally. Since the function $g_1(\tau)$ exhibits connections with $E_4(\tau)$, the theta function, and others, I kept computing the coefficients out of curiosity. Through this process, I noticed that the coefficients of $j(\tau)$ might be expressible in terms of $\mathbf{t}_1(d)$, and after some trial and error, I eventually arrived at the formula." Once this observation is made, the result can be proved (albeit in a somewhat ad hoc way) by applying Zagier's modularity theorem (\cref{thm:Zagier-gen}). In addition, as Kaneko~\cite{Kaneko1996} also pointed out, by applying the relation
\[
	\mathbf{t}_2(d) = \mathbf{t}_1(4d) + \left(\frac{-d}{2}\right) \mathbf{t}_1(d) + 2\mathbf{t}_1(d/4),
\]
which is obtained via the action of Hecke operators, one sees that this theorem is equivalent to \cref{thm:Kaneko-formula}. From the perspective of a natural proof, the formulation in term of $\mathbf{t}_2(d)$ is more suitable.

\begin{proof}[Proof of \cref{thm:Kaneko-formula}]
	Applying \cref{prop:RC-bracket}, we have
	\[
		F_{g_2, 0}(\tau) = \sum_{n=-1}^\infty \left(\sum_{r \in \Z} \mathbf{t}_2(4n-r^2) \right) q^n = -2 q^{-1} + O(q) \in M_2^!(\SL_2(\Z)).
	\]
	On the other hand, since $D = q \frac{\dd}{\dd q}$, we observe that 
	\[
		Dj(\tau) = -q^{-1} + \sum_{n=1}^\infty n \mathrm{Coeff}_{q^n}(j) q^n \in M_2^!(\SL_2(\Z)).
	\]
	We therefore conclude that $F_{g_2, 0} - 2Dj \in M_2(\SL_2(\Z)) = \{0\}$, which yields Kaneko's formula.
\end{proof}

\section{Developments after Zagier: A perspective from Kaneko's formula}

Zagier~\cite{Zagier2002} suggested several directions for generalizing \cref{thm:Zagier-gen}, including twisting by characters, replacing $\SL_2(\Z)$ with congruence subgroups, and considering modular forms of nonzero weight. As an early development along these lines, C.~H.~Kim~\cite{Kim2004, Kim2006} gave extensions to Hauptmoduln of higher levels. These ideas were later unified and extended using the theory of theta lifts by Bruinier--Funke~\cite{BruinierFunke2006} and others, which initiated a series of further developments.

\subsection{Higher level analogues of Kaneko's formula}

We first briefly review the results of Bruinier--Funke~\cite{BruinierFunke2006} and then present several higher-level analogues of Kaneko's formula that have arisen from this framework.

We first recall an analogue of \cref{thm:Zagier-gen} for the Kronecker--Hurwitz class numbers $\mathbf{t}_0(d)$, established by Zagier~\cite{Zagier1975} in 1975.

\begin{theorem}[Zagier~\cite{Zagier1975}]\label{thm:Zagier-KH-3/2}
	The generating series
	\[
		\widehat{g}_0(\tau) \coloneqq -\frac{1}{12} + \sum_{0 < d \equiv 0, 3\ (4)} \mathbf{t}_0(d) q^d + \frac{1}{8\pi \sqrt{v}} + \frac{1}{4\sqrt{\pi}} \sum_{n=1}^\infty n \Gamma \left(-\frac{1}{2}; 4\pi n^2 v\right) q^{-n^2}
	\]
	is a harmonic Maass form of weight $3/2$ on $\Gamma_0(4)$, where we set $\tau = u + iv$ with $u, v \in \R$, and 
	\[
		\Gamma(s; x) \coloneqq \int_x^\infty e^{-t} t^{s-1} \dd t
	\]
	denotes the incomplete Gamma function.
\end{theorem}

Bruinier--Funke unifies the two theorems of Zagier and developed powerful techniques applicable to modular forms of weight $0$ for more general $\Gamma$, beyond $\SL_2(\Z)$. In particular, their framework provides geometric interpretations for the values of $\mathbf{t}_m(d)$ with $d \le 0$, which were previously defined in an ad hoc manner. Roughly speaking, they constructed a map (theta lift) sending weight 0 modular forms to modular forms of weight $3/2$, and showed that, as specific examples, the constant function $1$ maps to $\widehat{g}_0(\tau)$, while the modular function $j_m(\tau)$ maps to a form corresponding to $g_m(\tau, z)$.

Their main theorem~\cite[Theorem 4.5]{BruinierFunke2006} is not formulated in terms of Jacobi forms, but rather in the language of vector-valued modular forms associated with the Weil representation, (see also~\cite{Bruinier2002}). Here, we do not go into the full details, but instead present only the special case in terms of Jacobi forms.

Let $p$ be a prime number. For a negative discriminant $-d \equiv 0, 1 \pmod{4}$ and an integer $h \pmod{2p}$ with $h^2 \equiv -d \pmod{4p}$, we define the subsets of $\mathcal{Q}_d$ by
\begin{align*}
	\mathcal{Q}_{d,p} &\coloneqq \{[a, b, c] \in \mathcal{Q}_d : a \equiv 0 \pmod{p}\},\\
	\mathcal{Q}_{d,p,h} &\coloneqq \{[a, b, c] \in \mathcal{Q}_{d,p} : b \equiv h \pmod{2p}\}.
\end{align*}
Let $\Gamma_0^*(p) \subset \SL_2(\R)$ denote the \emph{Fricke group}, generated by $\Gamma_0(p)$ together with the Fricke involution $W_p = \smat{0 & -1/\sqrt{p} \\ \sqrt{p} & 0} \in \SL_2(\R)$. Then $\Gamma_0(p)$ acts on $\mathcal{Q}_{d,p,h}$ while the Fricke group $\Gamma_0^*(p)$ acts on $\mathcal{Q}_{d,p}$. For a weight $0$ weakly holomorphic modular form $f \in M_0^!(\Gamma_0^*(p))$, we define two types of modular traces:
\begin{align*}
	\mathbf{t}_f(d) &\coloneqq \sum_{Q \in \mathcal{Q}_{d,p,h}/\Gamma_0(p)} \frac{1}{|\overline{\Gamma_0(p)}_Q|} f(\alpha_Q),\\
	\mathbf{t}_f^*(d) &\coloneqq \sum_{Q \in \mathcal{Q}_{d,p}/\Gamma_0^*(p)} \frac{1}{|\overline{\Gamma_0^*(p)}_Q|} f(\alpha_Q),
\end{align*}
where $\overline{\Gamma} = \Gamma/\{\pm I\}$. While the definition of $\mathbf{t}_f(d)$ appears to depend on the choice of representatives $h \pmod{2p}$, the value is in fact independent of this choice, thanks to the invariance of $f$ under $\Gamma_0^*(p)$. More precisely, we have the following.

\begin{lemma}\label{lem:tf-f*}
	For $f \in M_0^!(\Gamma_0^*(p))$, we have $\mathbf{t}_f(d) = 2^{\omega(\gcd(p, d))} \mathbf{t}_f^*(d)$.
\end{lemma}

\begin{proof}
	For each $h$ satisfying $h^2 \equiv -d \pmod{4p}$, we see that $p \mid h$ if and only if $p \mid d$. We note that $[a,b,c] \circ W_p = [pc, -b, a/p]$. We then verify the following:
	\begin{enumerate}
		\item If $p \nmid d$, then the map $\mathcal{Q}_{d,p,h}/\Gamma_0(p) \ni [a,b,c] \mapsto [a,b,c] \in \mathcal{Q}_{d,p}/\Gamma_0^*(p)$ is bijective. Thus, we have $\mathbf{t}_f(d) = \mathbf{t}_f^*(d)$.
		\item If $p \mid d$, for each $Q = [a,b,c] \in \mathcal{Q}_{d,p,h}$, there are two possibilities.
		\begin{itemize}
		 	\item If $Q \neq Q \circ W_p$ in $\mathcal{Q}_{d,p,h}/\Gamma_0(p)$, then the map $\mathcal{Q}_{d,p,h}/\Gamma_0(p) \ni Q, Q \circ W_p \mapsto Q \in \mathcal{Q}_{d,p}/\Gamma_0^*(p)$ is 2-1 correspondence.
		\item If $Q = Q \circ W_p$ in $\mathcal{Q}_{d,p,h}/\Gamma_0(p)$, then we have $|\overline{\Gamma_0^*(p)}_Q| = 2|\overline{\Gamma_0(p)}_Q|$.
		\end{itemize}
		Thus, we have $\mathbf{t}_f(d) = 2 \mathbf{t}_f^*(d)$.
	\end{enumerate}
	This concludes the proof.
\end{proof}

In this setting, the following was shown by Bruinier--Funke.

\begin{theorem}[Bruinier--Funke~\cite{BruinierFunke2006}, see also C.~H.~Kim~\cite{Kim2007}]
	For $f(\tau) = \sum_{n \gg -\infty} a_n q^n \in M_0^!(\Gamma_0^*(p))$ with $a_0 = 0$, we put
	\[
		\mathbf{t}_f(0) = 2 \sum_{n=1}^\infty (\sigma_1(n) + p \sigma_1(n/p)) a_{-n},
	\]
	and for negative $d$,
	\[
		\mathbf{t}_f(d) = \begin{cases}
			- 2^{\omega(\gcd(p, \kappa))} \kappa \sum_{\kappa \mid m} a_{-m} &\text{if } d = -\kappa^2 \text{ for some positive integer $\kappa$},\\
			0 &\text{otherwise}.
		\end{cases}
	\]
	Then we have
	\[
		\phi_f(\tau, z) = \sum_{\substack{n \gg -\infty \\ r \in \Z}} \mathbf{t}_f(4pn-r^2) q^n \zeta^r \in J_{2,p}^!.
	\]
\end{theorem}

By applying the discussion from the previous sections, we obtain higher-level analogues of Kaneko's formula. As an example, we illustrate this by considering the case $p=2$ with the Hauptmodul for $\Gamma_0^*(2)$ defined by
\[
	j_2^*(\tau) = \left(\frac{\eta(\tau)}{\eta(2\tau)}\right)^{24} + 24 + 2^{12} \left(\frac{\eta(2\tau)}{\eta(\tau)}\right)^{24} = q^{-1} + 4372q + 96256 q^2 + \cdots.
\]
We also consider the polynomial $j_{2,m}^* \in \Z[j_2^*]$ whose Fourier expansion begins $j_{2,m}^*(\tau) = q^{-m} + O(q)$. As in~\cref{cor:Kaneko-Zagier-rel}, a recurrence relation characterizing $\mathbf{t}_m^{(2)}(d) \coloneqq \mathbf{t}_{j_{2,m}^*}(d)$ can be derived. (An equivalent form of the recurrence relation was also obtained by C.~H.~Kim~\cite{Kim2008}).

\begin{corollary}[The case of $m=1$]\label{cor:t12-rec}
	Recall that $\mathbf{t}_1^{(2)}(0) = 2$ and $\mathbf{t}_1^{(2)}(-1) = -1$.
	For any integer $n \ge 1$, we have
	\begin{align*}
		\sum_{r \in \Z} \mathbf{t}_1^{(2)}(8n-r^2) &= 0,\\
		\sum_{r \in \Z} (r^2 - 2n) \mathbf{t}_1^{(2)}(8n-r^2) &= -480 \sigma_3(n),\\
		\sum_{r \in \Z} (r^4 - 6nr^2 + 4n^2) \mathbf{t}_1^{(2)}(8n-r^2) &= 1008 \sigma_5(n).
	\end{align*}
\end{corollary}

\begin{proof}
	This follows from \cref{prop:RC-bracket} and \cref{prop:RC-expansion} for $\nu = 0, 1, 2$.
\end{proof}

This recurrence allows us to compute the first (non-zero) few examples as follows:
\begin{align*}
	\mathbf{t}_1^{(2)}(4) = -52, \quad \mathbf{t}_1^{(2)}(7) = -23, \quad \mathbf{t}_1^{(2)}(8) = 152, \quad \mathbf{t}_1^{(2)}(12) = -496.
\end{align*}
In her master's thesis, Ohta~\cite{Ohta2009} made numerical observations similar to Kaneko's, based on a table of the values $\mathbf{t}_1^{(2)}(d)$, and obtained level $2$ (and other higher level) analogues of \cref{thm:Kaneko-original}. Here, we present a simpler expression corresponding to \cref{thm:Kaneko-formula}, which also refines the proof of results from the master's thesis of Osanai~\cite{MatsusakaOsanai2017}, written in collaboration with the author. Let 
\[
	\mathbf{t}_2^{(2*)}(d) \coloneqq \mathbf{t}_{j_{2,2}^*}^*(d) = 2^{-\omega(\gcd(2,d))} \mathbf{t}_2^{(2)}(d)
\]
by \cref{lem:tf-f*}. As in \cref{cor:t12-rec}, one obtains recurrence relations that allow the computation of $\mathbf{t}_2^{(2)}(d)$, with the initial values:
\[
\begin{array}{c|cccccc}
d & -4 & -1 & 0 & 4 & 7 & 8 \\ \hline
\mathbf{t}_2^{(2)}(d) & -4 & -1 & 10 & 1036 & -8215 & 14360 \\
\mathbf{t}_2^{(2*)}(d) & -2 & -1 & 5 & 518 & -8215 & 7180
\end{array}
\]

\begin{theorem}[\cite{Matsusaka2017}]\label{thm:higher-level}
	Let $j_2(\tau)$ be the Hauptmodul of $\Gamma_0(2)$ defined by
	\[
		j_2(\tau) = \left(\frac{\eta(\tau)}{\eta(2\tau)}\right)^{24} + 24 = q^{-1} + 276q - 2048 q^2 + 11202q^3 + \cdots.
	\]
	Then, for any $n \ge 1$, we have
	\[
		2n \mathrm{Coeff}_{q^n}(j_2) = \sum_{r \in \Z} \mathbf{t}_2^{(2*)}(4n-r^2) + 24 \sigma_1^{(2)}(n),
	\]
	where we define $\sigma_1^{(2)}(n) = \sum_{\substack{d \mid n \\ d \not\equiv 0\ (2)}} d$.
\end{theorem}

\begin{proof}
	While the proof is essentially the same as in~\cite{Matsusaka2017}, a slight shortcut has been included. By applying \cref{prop:Gphi} and \cref{prop:Gphi-Fourier} with \cref{ex:m=p} to $\phi = \phi_{j_{2,2}^*}$, we have
	\[
		G_{\phi_{j_{2,2}^*}}(\tau) = 2 \sum_{n=-1}^\infty \left(\sum_{r \in \Z} \mathbf{t}_2^{(2*)}(4n-r^2) \right) q^n = -4q^{-1} - 2 + O(q)  \in M_2^!(\Gamma_0(2)).
	\]
	Moreover, since
	\begin{align*}
		\tau^{-2} G_{\phi_{j_{2,2}^*}}\left(-\frac{1}{\tau}\right) &= \frac{1}{4} \left(\phi_{j_{2,2}^*} \left(\frac{\tau}{2}, 0 \right) + \phi_{j_{2,2}^*} \left(\frac{\tau}{2}, \frac{1}{2}\right) \right)\\
		&= \frac{1}{4} \sum_{n=0}^\infty \sum_{r \in \Z} (1+(-1)^r) \mathbf{t}_2^{(2)}(8n-r^2) q^{n/2},
	\end{align*}
	we see that $G_{\phi_{j_{2,2}^*}}(\tau)$ has no pole at the cusp $0$. Therefore, $G_{\phi_{j_{2,2}^*}}(\tau) - 4Dj_2(\tau) \in M_2(\Gamma_0(2))$, which is generated by the Eisenstein series
	\[
		E_2^{(2)}(\tau) \coloneqq 1 + 24 \sum_{n=1}^\infty \sigma_1^{(2)}(n) q^n.
	\]
	Thus, we obtain $G_{\phi_{j_{2,2}^*}}(\tau) = 4Dj_2(\tau) - 2E_2^{(2)}(\tau)$, which concludes the proof.
\end{proof}

In the author's earlier work~\cite{MatsusakaOsanai2017, Matsusaka2017}, similar techniques were employed to obtain analogues of Kaneko's formula in several genus zero cases of $\Gamma_0(N)$.

We also highlight several subsequent developments that build on the foundational work of Bruinier--Funke and are motivated by directions suggested by Zagier. The theta lift involving twists was constructed by Alfes--Ehlen~\cite{AlfesEhlen2013} and later extended to modular forms (and harmonic Maass forms) of nonzero weight by Alfes--Schwagenscheidt~\cite{AlfesSchwagenscheidt2018, AlfesSchwagenscheidt2021}. In parallel with this line of work, more direct approaches that do not rely on the language of theta lifts have also been developed. A notable example is the work of Duke--Imamo\={g}lu--T\'{o}th~\cite{DukeImamogluToth2011}, who not only treated singular moduli but also, importantly, extended their methods to cycle integrals. (This result can also be interpreted in terms of theta lifts~\cite{BruinierFunkeImamoglu2015}). The cycle integrals are closely related to Kaneko's other important discovery concerning the $j$-function, namely the val-function~\cite{Kaneko2009}. Building on their methods, the author extended Zagier's result to polyharmonic Maass forms in the case of $\SL_2(\Z)$, as presented in~\cite{Matsusaka2019, Matsusaka2019-phd}.

\subsection{Eichler--Selberg relations}

In analogy with how Kaneko's formula (\cref{thm:Kaneko-formula}) follows from Zagier's result (\cref{thm:Zagier-gen}), one may ask whether the Hurwitz relation (\cref{thm:Hurwitz}) can similarly be derived from another result of Zagier (\cref{thm:Zagier-KH-3/2}) concerning the Kronecker--Hurwitz class numbers. The answer is yes: such a derivation was given by Zagier himself~\cite{Zagier1991}, and a more general treatment was later developed by Mertens~\cite{Mertens2014}. Since the generating function $\widehat{g}_0(\tau)$ in \cref{thm:Zagier-KH-3/2} is not a holomorphic modular form, suitable modifications are necessary. Nevertheless, by working with the real-analytic Eisenstein series and applying techniques such as Rankin--Cohen brackets and holomorphic projection, (see~\cite[Chapter 10]{BFOR2017}), an argument analogous to \cref{prop:RC-bracket} can be carried out. In particular, the Hurwitz relation is obtained as the case corresponding to $\nu = 0$. This approach also yields further relations for the class numbers $\mathbf{t}_0(d)$ for general $\nu \ge 0$, analogous to \cref{cor:Kaneko-Zagier-rel} and \cref{cor:t12-rec}. These results are referred to as the \emph{Eichler--Selberg trace formula}. We also note that these derivations are summarized by Ono--Saad~\cite{OnoSaad2022}.

\begin{theorem}[Eichler--Selberg trace formula~\cite{Eichler1955, Selberg1956}]\label{thm:ES-trace}
	Let $p_l(r,n) \coloneqq p_{2,l}(r,n)$ be the polynomials defined as in \cref{prop:RC-expansion}. For any positive integers $\nu$ and $n$, we have
	\[
		\sum_{r \in \Z} p_{2\nu}(r,n) \mathbf{t}_0(4n-r^2) + \sum_{d \mid n} \min \left(d, \frac{n}{d}\right)^{2\nu+1} = -2 \sum_{j=1}^{d_{2\nu+2}} c_j(n),
	\]
	where $\{f_j\}_{j=1}^{d_{2\nu+2}}$ is an orthogonal basis (a Hecke basis) of the space of cusp forms $S_{2\nu+2}(\SL_2(\Z))$ consisting of normalized Hecke eigenforms, and $c_j(n)$ denotes the $n$-th Fourier coefficient of $f_j(\tau)$.
\end{theorem}

\begin{example}
	When $\nu=5$, we have
	\[
		\sum_{r \in \Z} (r^{10} - 9nr^8 + 28n^2 r^6 - 35 n^3r^4 + 15n^4r^2 - n^5) \mathbf{t}_0(4n-r^2) + \sum_{d \mid n} \min \left(d, \frac{n}{d}\right)^{11} = -2 \tau(n),
	\]
	where $\tau(n)$ is the Ramanujan $\tau$-function.
\end{example}

The right-hand side (up to a factor of $-2$) gives the trace of the Hecke operator $T_n$ acting on $S_{2\nu+2}(\SL_2(\Z))$. In light of this theorem and \cref{cor:Kaneko-Zagier-rel}, it is natural to ask how the sum
\[
	\sum_{r \in \Z} p_{2\nu}(r,n) \mathbf{t}_m(4n-r^2)
\]
can be described for $m > 0$. Indeed, Kaneko~\cite{Kaneko2001} himself raises a related question, referring to the Hurwitz relation (\cref{thm:Hurwitz}) and noting that ``many such relations satisfied by the class number of quadratic forms are known, and numerous formulas are collected in Dickson's book~\cite[Chapter VI]{Dickson1966}. It is possible that various variants of recurrence relations for $\mathbf{t}_1(d)$ corresponding to them also exist".

For instance, applying \cref{prop:RC-bracket} and \cref{prop:RC-expansion} for $\nu = 5$, we have
\begin{align*}
	F_{g_1, 5}(\tau) &= \sum_{n=0}^\infty \left(\sum_{r \in \Z} p_{10}(r, n) \mathbf{t}_1(4n-r^2) \right) q^n\\
		&= -2 + 48q -394272q^2 + \cdots\\
		&= -2 \left(E_{12}(\tau) - \frac{82104}{691} \Delta(\tau) \right) \in M_{12}(\SL_2(\Z)).
\end{align*}
The interpretation of the rational coefficient of $\Delta(\tau)$ becomes an issue. Our theorem in~\cite{DengMatsusakaOno2024} shows that it can be described in terms of the \emph{shifted convolution $L$-function} defined by
\[
	\widehat{L}(f, m; s) \coloneqq \sum_{n=1}^\infty \frac{c_f(n) c_f(n+m)}{n^s} - \sum_{n=m+1}^\infty \frac{c_f(n)c_f(n-m)}{n^s}
\]
for a cusp form $f(\tau) = \sum_{n=1}^\infty c_f(n) q^n$ on $\SL_2(\Z)$. In the present case, we find that
\[
	\frac{82104}{691} = 24 - \frac{\Gamma(11)}{(4\pi)^{11}} \frac{\widehat{L}(\Delta, 1; 11)}{\|\Delta\|^2},
\]
where $\|\Delta\|^2 = 0.0000010353\dots$ is the Petersson norm. In general, we have the following.

\begin{theorem}[{Deng--Matsusaka--Ono~\cite{DengMatsusakaOno2024}}]\label{thm:Eichler-Selberg-rel}
	For any positive integers $m, \nu > 0$, we have
	\begin{align*}
		&\sum_{n \gg -\infty} \left(\sum_{r \in \Z} p_{2\nu}(r,n) \mathbf{t}_m(4n-r^2) \right) q^n\\
			&= -2 \left( \sum_{\kappa \mid m} \sum_{0 < r \le \kappa} r^{2\nu+1} P_{2\nu+2, -r(\kappa-r)}(\tau) - \sum_{j=1}^{d_{2\nu+2}} \left(24 \sigma_1(m) - \frac{\Gamma(2\nu+1)}{(4\pi)^{2\nu+1}} \frac{\widehat{L}(f_j, m; 2\nu+1)}{\|f_j\|^2} \right) f_j(\tau) \right),
	\end{align*}
	where $P_{k,m}(\tau)$ is the Eisenstein/Poincar\'{e} series defined by
	\[
		P_{k,m}(\tau) \coloneqq \sum_{\gamma \in \SL_2(\Z)_\infty \backslash \SL_2(\Z)} q^m |_k \gamma,
	\]
	and $\{f_j\}_{j=1}^{d_{2\nu+2}}$ is a Hecke basis of $S_{2\nu+2}(\SL_2(\Z))$ as in \cref{thm:ES-trace}. We note that the index of $p_{2\nu}(r,n)$ differs by $2$ from the one used in~\cite{DengMatsusakaOno2024}.
\end{theorem}

The shifted convolution $L$-function was originally considered by Selberg~\cite{Selberg1965}, although it was not the main focus of his work and appears only briefly at the end of his article. This object was later studied by Hoffstein--Hulse~\cite{HoffsteinHulse2016}, Mertens--Ono~\cite{MertensOno2016}, and others. While we do not provide a proof of \cref{thm:Eichler-Selberg-rel} in this article, the general strategy follows Zagier's approach to the Eichler--Selberg trace formula given in~\cite{Zagier1991} and relies on techniques such as the Rankin--Cohen bracket and the Rankin--Selberg method.

\subsection{Other variants}

Lin--McSpirit--Vishnu~\cite{LinMcSpiritVishnu2020} discovered a formula analogous to Kaneko's formula, involving the classical partition function $p(n)$ and its variant $\mathrm{spt}(n)$. The function $\mathrm{spt}(n)$ counts the total number of smallest parts across all partitions of $n$. The abbreviation ``$\mathrm{spt}$" stands for ``smallest parts". To illustrate, consider $n=4$:
\[
	\underline{4}, \quad 3 + \underline{1}, \quad \underline{2} + \underline{2}, \quad 2 + \underline{1} + \underline{1}, \quad \underline{1} + \underline{1} + \underline{1} + \underline{1},
\]
where the smallest part in each partition is underlined. There are ten such occurrences in total, hence $\mathrm{spt}(4) = 10$. This function $\mathrm{spt}(n)$ was introduced by Andrews~\cite{Andrews2008}, and the following generating function formula is known:
\[
	\sum_{n=1}^\infty \mathrm{spt}(n) q^n = \prod_{n=1}^\infty \frac{1}{1-q^n} \cdot \sum_{n=1}^\infty \frac{q^n \prod_{m=1}^{n-1} (1-q^m)}{1-q^n} = q + 3q^2 + 5q^3 + 10q^4 + 14q^5 + \cdots.
\]

\begin{theorem}[Lin--McSpirit--Vishnu~\cite{LinMcSpiritVishnu2020}]\label{thm:LSV}
	For a positive integer $n$, we define
	\[
		m(24n-1) \coloneqq \mathrm{spt}(n) + \frac{24n-1}{12} p(n),
	\]
	and set $m(d) = 0$ for all other $d \in \Z$. For a prime $\ell \ge 5$, we define
	\[
		m_\ell(n) \coloneqq m(\ell^2 n) + \left(\frac{3}{\ell}\right) \left( \left(\frac{-n}{\ell}\right) - (1+\ell) \right) m(n) + \ell m(n/\ell^2).
	\]
	Additionally, we set $m_\ell(-\ell^2) = -\ell/12$ and $m_\ell(-1) = \left(\frac{3}{\ell}\right) \frac{\ell}{12}$. Then, the Fourier coefficients of the $j$-function satisfy
	\[
		\mathrm{Coeff}_{q^n}(j) = \frac{6}{5n} \sum_{r \in \Z} \left(\frac{12}{r}\right) m_5(24n-r^2).
	\]
\end{theorem}

Although this statement is not explicitly stated in~\cite{LinMcSpiritVishnu2020}, it is equivalent to their Theorem~1.2. Their theorem serves as an analogue of Kaneko's formula (\cref{thm:Kaneko-original}), and just as Kaneko's formula can be reformulated as in \cref{thm:Kaneko-formula}, their theorem can similarly be rewritten as in the theorem above. In the final remark of Section~1 in~\cite{LinMcSpiritVishnu2020}, the authors announced that ``Toshiki Matsusaka informed the authors that he has obtained further similar results", which refers to \cref{thm:LSV} and the following proof.

\begin{proof}
	As in Zagier's theorem (\cref{thm:Zagier-gen}), the first goal is to construct a suitable Jacobi form. Let
	\[
		M_\ell(\tau) \coloneqq \sum_{n = -\ell^2}^\infty m_\ell(n) q^n.
	\]
	Recall that Ono~\cite[Theorem 2.2]{Ono2011} showed that the function
	\[
		F_\ell(\tau) \coloneqq \eta(\tau)^{\ell^2} M_\ell(\tau/24)
	\]
	is a weight $(\ell^2+3)/2$ holomorphic modular form on $\SL_2(\Z)$. Since $\ell^2 \equiv 1 \pmod{24}$, the function $N_\ell(\tau) \coloneqq M_\ell(\tau/24)$ satisfies the transformation laws
	\begin{align}\label{eq:Nl-trans}
		N_\ell(\tau+1) = \mathbf{e}\left(-\frac{1}{24}\right) N_\ell(\tau), \qquad N_\ell \left(-\frac{1}{\tau}\right) = i^{1/2} \tau^{3/2} N_\ell(\tau).
	\end{align}
	For each $\mu \in \Z/12\Z$, we define
	\[
		h_\mu(\tau) \coloneqq \left(\frac{12}{\mu}\right) N_\ell(\tau).
	\]
	When $(\mu,12) > 1$, we have $h_\mu(\tau) = 0$. To match the transformation law for Jacobi forms given in~\eqref{eq:vect-modular}, we rewrite \eqref{eq:Nl-trans} using the identity $\sum_{\nu \in \Z/12\Z} \mathbf{e}(\mu \nu/12) \left(\frac{12}{\mu \nu}\right) = \sqrt{12}$ for $(\mu,12) = 1$ as follows:
	\[
		h_\mu(\tau+1) = \mathbf{e}\left(-\frac{\mu^2}{24}\right) h_\mu(\tau), \qquad h_\mu\left(-\frac{1}{\tau}\right) = \frac{i^{1/2} \tau^{3/2}}{\sqrt{12}} \sum_{\nu \in \Z/12\Z} \mathbf{e} \left(\frac{\mu \nu}{12}\right) h_{\nu}(\tau).
	\]
	This is the transformation law that components of Jacobi forms of weight $2$ and index $6$ should satisfy. Therefore, we have
	\begin{align*}
		\varphi_\ell(\tau, z) &\coloneqq \sum_{\mu \in \Z/12\Z} h_\mu(\tau) \theta_{6, \mu}(\tau, z) = N_\ell(\tau) \sum_{\mu \in \Z/12\Z} \left(\frac{12}{\mu}\right) \theta_{6,\mu}(\tau, z)\\
			&= \sum_{n \gg -\infty} \sum_{r \in \Z} \left(\frac{12}{r}\right) m_\ell(24n-r^2) q^n \zeta^r \in J_{2,6}^!.
	\end{align*}
	For $\ell = 5$, evaluating at $z =0$ implies that
	\[
		\varphi_5(\tau, 0) = \sum_{n \gg -\infty} \left(\sum_{r \in \Z} \left(\frac{12}{r}\right) m_5(24n-r^2) \right) q^n = -\frac{5}{6} q^{-1} + O(q) \in M_2^!(\SL_2(\Z)),
	\]
	which coincides with $5/6 Dj(\tau)$.
\end{proof}

Other examples in this direction include the work of Imamo\={g}lu--Raum--Richter~\cite{ImamogluRaumRichter2014}, who provided an identity expressing the finite sums of the Fourier coefficients of Ramanujan's 3rd-order mock theta function
\[
	f_3(q) \coloneqq 1 + \sum_{n=1}^\infty \frac{q^{n^2}}{(1+q)^2 (1+q^2)^2 \cdots (1+q^n)^2}.
\]
Similar techniques can, in fact, be applied to various other combinatorial $q$-series as well.

\subsection{Applications}

We also briefly introduce two results as applications of the Kaneko formula. First, according to a result published in 2003 by Baier--K\"{o}hler~\cite{BaierKohler2003}, there are several methods for computing the Fourier coefficients of the $j$-function. Among them, the method that combines the recurrence relations satisfied by the traces of singular moduli (\cref{cor:Kaneko-Zagier-rel}) with Kaneko's formula (\cref{thm:Kaneko-original}) is particularly efficient. 

Second, Murty--Sampath~\cite{MurtySampath2016} provided an alternative proof for the asymptotic formula for the Fourier coefficients of the $j$-function by applying the Laplace method to the Kaneko formula, thereby avoiding the use of the circle method. We now briefly outline their approach. The Laplace method is stated as follows.

\begin{lemma}[Laplace's method]
	Let $f$ be a real-valued $C^2$ function defined on an open interval $(a,b) \subset \R$. Suppose that $f$ has a unique maximum at $t = c \in (a,b)$ so that $f''(c) < 0$. Then, as $\lambda \to \infty$, we have
	\[
		\int_a^b e^{\lambda f(t)} \dd t \sim e^{\lambda f(c)} \left(\frac{2\pi}{\lambda |f''(c)|}\right)^{1/2}.
	\]
\end{lemma}

This leads to the following well-known asymptotic formula.

\begin{theorem}
	We have
	\[
		\mathrm{Coeff}_{q^n}(j) \sim \frac{e^{4\pi \sqrt{n}}}{\sqrt{2} n^{3/4}}
	\]
	as $n \to \infty$.
\end{theorem}

\begin{proof}
	Here, we sketch the proof. For each quadratic form $Q = [a,b,c] \in \mathcal{Q}_d$ such that the associated root $\alpha_Q \in \bbH$ lies in the standard fundamental domain of $\SL_2(\Z) \backslash \bbH$, we have
	\[
		j_2(\alpha_Q) = e^{-4\pi i \frac{-b+ i\sqrt{d}}{2a}} + O(e^{-\pi \sqrt{d}/a}).
	\]
	For each discriminant $-d < 0$, the contribution of $j_2(\alpha_Q)$ is largest when $a=1$. It is known, however, that any quadratic form $[1, b, c]$ is $\SL_2(\Z)$-equivalent to either $[1,0,d/4]$ if $d \equiv 0 \pmod{4}$, or $[1,1,(1+d)/4]$ if $d \equiv 3 \pmod{4}$. Therefore, the contribution to $\mathbf{t}_2(d)$ comes only from these special quadratic forms, so that we have $\mathbf{t}_2(d) \sim e^{2\pi \sqrt{d}}$ as $d \to \infty$. (The refined asymptotic behavior of $\mathbf{t}_m(d)$ is given in the works of Duke and Andersen~\cite{Duke2006, AndersenDuke2022}). Therefore, Kaneko's formula (\cref{thm:Kaneko-formula}) implies that
	\[
		\mathrm{Coeff}_{q^n}(j) \sim \frac{1}{2n} \sum_{|r| \le 2\sqrt{n}} e^{2\pi \sqrt{4n-r^2}} = \frac{1}{\sqrt{n}} \frac{1}{2\sqrt{n}} \sum_{|r| \le 2\sqrt{n}} e^{4\pi \sqrt{n} \sqrt{1 - (\frac{r}{2\sqrt{n}})^2}}.
	\]
	Interpreting this sum as a Riemann sum, it can be approximated as
	\[
		\mathrm{Coeff}_{q^n}(j) \sim \frac{1}{\sqrt{n}} \int_{-1}^1 e^{4\pi \sqrt{n} \sqrt{1-t^2}} \dd t.
	\]
	By applying Laplace's method to the function $f(t) = 4\pi \sqrt{1-t^2}$ with $\lambda = \sqrt{n}$ and $c=0$, we have
	\[
		\mathrm{Coeff}_{q^n}(j) \sim \frac{1}{\sqrt{n}} e^{4\pi \sqrt{n}} \left(\frac{2\pi}{4\pi \sqrt{n}}\right)^{1/2} = \frac{e^{4\pi \sqrt{n}}}{\sqrt{2} n^{3/4}}
	\]
	as desired. For further details, see~\cite{MurtySampath2016}.
\end{proof}

In~\cite{MatsusakaOsanai2017}, this method is further applied to higher-level analogues such as~\cref{thm:higher-level}, yielding asymptotic formulas for the coefficients of several Hauptmoduln as well.

\section{Rohrlich-type divisor sums}\label{sec:Rohrlich}

Finally, we review some results on more general sums of special values of modular functions over divisors, extending beyond traces of singular moduli. 

\subsection{Bruinier--Kohnen--Ono's formula}

Let $N$ be a positive integer, and set $Y_0(N) \coloneqq \Gamma_0(N) \backslash \bbH$. The $\Q$-vector space of divisors on $Y_0(N)$ is given by
\[
	\mathrm{Div}(Y_0(N))_\Q = \left\{ \sum_{[z] \in Y_0(N)} n_z [z] : n_z \in \Q, n_z = 0 \text{ for almost all } [z] \right\}.
\]

\begin{definition}
	For a function $F: Y_0(N) \to \C$, we define a map $\mathcal{D}_F: \mathrm{Div}(Y_0(N))_\Q \to \C$ by
	\[
		\mathcal{D}_F(D) \coloneqq \sum_{[z] \in Y_0(N)} n_z F(z)
	\]
	for $D = \sum_{[z] \in Y_0(N)} n_z [z]$.
\end{definition}

\begin{example}
	For $N=1$ and $F(\tau) = j_m(\tau)$ with $m \ge 1$, the trace of singular moduli $\mathbf{t}_m(d)$ is given by $\mathcal{D}_F(D)$ for
	\begin{align}\label{eq:disc-div}
		D = \sum_{Q \in \mathcal{Q}_d/\PSL_2(\Z)} \frac{1}{|\PSL_2(\Z)_Q|} [\alpha_Q] \in \mathrm{Div}(Y_0(1))_\Q.
	\end{align}
\end{example}

As a different type of sum, we now consider those involving divisors of meromorphic modular forms. Let $M_*^\mathrm{mer}(\Gamma_0(N))$ be the space of meromorphic modular forms of integral weight on $\Gamma_0(N)$. For any nonzero $f \in M_*^\mathrm{mer}(\Gamma_0(N))$, we define the \emph{divisor map}
\[
	\mathrm{div}(f) \coloneqq \sum_{[z] \in Y_0(N)} \frac{\mathrm{ord}_z(f)}{|\overline{\Gamma_0(N)}_z|} [z] \in \mathrm{Div}(Y_0(N))_\Q,
\]
where $\mathrm{ord}_z(f) \in \Z$ is the order of $f$ as a meromorphic function on $\bbH$ at $\tau = z$. In this setting, Bruinier, Kohnen, and Ono established a different kind of relation connecting special values of modular functions with Fourier coefficients of modular forms.

\begin{theorem}[Bruinier--Kohnen--Ono~\cite{BruinierKohnenOno2004}]\label{thm:BKO}
	For a positive integer $m$ and $f \in M_k^\mathrm{mer}(\SL_2(\Z))$ of weight $k \in \Z$, we have
	\[
		\mathcal{D}_{j_m}(\div(f)) + 2k \sigma_1(m) = -\mathrm{Coeff}_{q^m} \left(\frac{Df}{f}\right).
	\]
\end{theorem}

\begin{example}
	Let $f(\tau) = E_4(\tau) \in M_4(\SL_2(\Z))$. This modular form has a unique simple zero at $\tau = e^{2\pi i/3}$ on $Y_0(1) = \SL_2(\Z) \backslash \bbH$. Applying the theorem to this case, we obtain
	\begin{align*}
		\frac{1}{3} j_m(e^{2\pi i/3}) + 8\sigma_1(m) &= -\mathrm{Coeff}_{q^m} \left(\frac{DE_4}{E_4}\right)\\
			&= -\mathrm{Coeff}_{q^m} \bigg(240q - 53280q^2 + 12288960q^3 + \cdots \bigg).
	\end{align*}
	For instance, this gives $\mathbf{t}_1(3) = -248$ and $\mathbf{t}_2(3) = 53256$. 
\end{example}

More generally, by applying \cref{thm:BKO} to the Borcherds product, one can also obtain a formula for $\mathbf{t}_m(d)$. See~\cite{Borcherds1995, Kaneko2001, Zagier2002} for details on the relation between $\mathbf{t}_m(d)$ and Borcherds' result. For each negative discriminant $-d \equiv 0, 1 \pmod{4}$, there exists a unique weakly holomorphic modular form $f_d(\tau)$ of weight $1/2$ whose Fourier expansion of the form
\[
	f_d(\tau) = q^{-d} + \sum_{0 < n \equiv 0, 1\ (4)} c_d(n) q^n.
\]
For example,
\begin{align*}
	f_3(\tau) &= \frac{1}{12} \frac{E_4(4\tau)}{\Delta(4\tau)} \left(6E_6(4\tau) D\theta_0(\tau) - \frac{1}{2} (DE_6)(4\tau) \cdot \theta_0(\tau)\right) - 88 \theta_0(\tau)\\
		&= q^{-3} - 248q + 26752q^4 - 85995q^5 + \cdots,
\end{align*}
where $\theta_0(\tau) = \sum_{r \in \Z} q^{r^2}$. Then the traces of singular moduli $\mathbf{t}_m(d)$ can be expressed in terms of the Fourier coefficients of $f_d(\tau)$. The following result was originally established by Zagier~\cite{Zagier2002}, and the proof presented below is merely a reformulation of his argument using \cref{thm:BKO}. 

\begin{corollary}[Zagier~\cite{Zagier2002}]
	For $m > 0$ and a negative discriminant $-d \equiv 0, 1 \pmod{4}$, we have
	\[
		\mathbf{t}_m(d) = \sum_{n|m} n c_d(n^2).
	\]
	In particular, $\mathbf{t}_1(d) = c_d(1)$.
\end{corollary}

\begin{proof}
	The Borcherds product $\Psi(f_d)$ is defined by the following expression and, moreover, satisfies the identity
	\[
		\Psi(f_d) \coloneqq q^{-\mathbf{t}_0(d)} \prod_{n=1}^\infty (1-q^n)^{c_d(n^2)} = \prod_{Q \in \mathcal{Q}_d/\Gamma} (j(\tau) - j(\alpha_Q))^{1/|\Gamma_Q|},
	\]
	where $\Gamma = \PSL_2(\Z)$ as before. This defines a weight $0$ modular form on $\SL_2(\Z)$ (with a certain character). In particular, $\Psi(f_d)$ has no poles in $\bbH$, and has a simple zero at each point $\alpha_Q$. Thus, we have
	\[
		\div(\Psi(f_d)) = \sum_{Q \in \mathcal{Q}_d/\Gamma} \frac{1}{|\Gamma_Q|} [\alpha_Q]
	\]
	as in \eqref{eq:disc-div}. Therefore, applying \cref{thm:BKO}, (which also applies to modular forms with character), yields that $\mathbf{t}_m(d) = \mathcal{D}_{j_m}(\div(\Psi(f_d)))$ is equal to the negative of the $m$-th coefficient of
	\begin{align*}
		D \log \Psi(f_d) &= -\mathbf{t}_0(d) - \sum_{n=1}^\infty n c_d(n^2) \frac{q^n}{1-q^n}\\
			&= -\mathbf{t}_0(d) - \sum_{\ell, n \ge 1} n c_d(n^2) q^{\ell n},
	\end{align*}
	which concludes the proof.
\end{proof}

This corollary shows that $\mathbf{t}_1(d)$ appears as a coefficient of a weight $1/2$ modular form $f_d$, while Zagier's formula \eqref{eq:Zagier-half} expresses $\mathbf{t}_1(d)$ as a coefficient of a weight $3/2$ modular form $g_1$. The symmetry between weights $1/2$ and $3/2$, known as the \emph{modular grid}, was formalized and established by Kaneko and Zagier~\cite{Zagier2002}. After numerous partial extensions, this symmetry has been significantly generalized to a much broader setting by Griffin--Jenkins--Molnar~\cite{GriffinJenkinsMolnar2022}.

\subsection{Jeon--Kang--Kim--Matsusaka's method}

Building on \cref{thm:BKO} as well as the classical result of Rohrlich~\cite{Rohrlich1984}, the author, together with Jeon--Kang--Kim~\cite{JKKM2025, JKKM2025-pre} investigated more general divisor sums of the form $\mathcal{D}_F(\div(f))$. Rohrlich's result corresponds to the special case where $F(\tau) = \log(\Im(\tau)^6 |\Delta(\tau)|)$, the so-called Kronecker limit function. Meanwhile, \cref{thm:BKO} corresponds to the case $F = j_m$, providing a parallel yet distinct example within this broader framework.

To motivate our approach, we recall two more classical results on $\SL_2(\Z)$. From now on, we write $\tau = u + iv$, where $u,v \in \R$ are the real and imaginary parts of $\tau \in \bbH$.

\begin{theorem}[Valence formula]\label{thm:valence}
	For the constant function $F = 1$ and for any nonzero $f \in M_k^\mathrm{mer}(\SL_2(\Z))$, we have
	\[
		\mathcal{D}_F(\div(f)) = \sum_{[z] \in \SL_2(\Z) \backslash \bbH} \frac{\mathrm{ord}_z(f)}{|\PSL_2(\Z)_z|} = \frac{k}{12} - \mathrm{ord}_{i\infty}(f),
	\]
	where $\mathrm{ord}_{i\infty}(f)$ is the order of $f$ at the cusp $i\infty$.
\end{theorem}

\begin{theorem}[Rohrlich~\cite{Rohrlich1984}]\label{thm:Rohrlich}
	For the Kronecker limit function $F(\tau) = \log(v^6 |\Delta(\tau)|)$ and for any $f \in M_k^{\mathrm{mer}}(\SL_2(\Z))$ with $\lim_{\tau \to i\infty}f(\tau) = 1$, we have
	\[
		\mathcal{D}_F(\div(f)) = -\frac{3}{\pi} \langle 1, \log(v^{k/2} |f|)\rangle^{\mathrm{reg}} + \frac{\pi k}{6}C - \frac{k}{2},
	\]
	where $\langle \cdot, \cdot \rangle^\mathrm{reg}$ is the regularized Petersson inner product and $C = (6 - 72\zeta'(-1) - 6 \log(4\pi))/\pi$ and $\zeta(s)$ is the Riemann zeta function.
\end{theorem}

In our article~\cite{JKKM2025}, we developed a unified framework for the Rohrlich-type divisor sums $\mathcal{D}_F(\div(f))$. To state the result, we define the hyperbolic Laplacian $\Delta_k$ defined by $\Delta_k \coloneqq -\xi_{2-k} \circ \xi_k$, where $\xi_k$ is the differential operator given by
\[
	\xi_k f(\tau) \coloneqq 2iv^k \overline{\frac{\partial}{\partial \overline{\tau}} f(\tau)}.
\]
Roughly speaking, our main result asserts the following.

\begin{theorem}[Jeon--Kang--Kim--Matsusaka~\cite{JKKM2025}]\label{thm:JKKM2025}
	Let $g$ be a smooth modular form of weight $2$ on $\Gamma_0(N)$, and let $f \in M_k^{\mathrm{mer}}(\Gamma_0(N))$. We define a weight $0$ smooth modular form $F(\tau) = \xi_2 g(\tau)$. Then we have
	\[
		\mathcal{D}_F(\div(f)) + \frac{1}{2\pi} \langle \Delta_0 F, \log(v^{k/2} |f|) \rangle^\mathrm{reg} = \sum_{\rho} C_\rho(f,g)
	\]
	if the regularized inner product converges. Here, the sum is taken over all cusps $\rho \in \Gamma_0(N) \backslash (\Q \cup \{i\infty\})$, and each constant $C_\rho(f,g) \in \C$ is explicitly determined by $f(\tau), g(\tau)$, and $\rho$.
\end{theorem}

\begin{proof}
	We leave the details to the original article~\cite{JKKM2025} and focus here on outlining the key idea. Let $\mathcal{D}_k$ be the differential operator defined by
	\[
		\mathcal{D}_k f(\tau) \coloneqq \frac{Df(\tau)}{f(\tau)} - \frac{k}{4\pi v}.
	\]
	A direct calculation shows that the $1$-form
	\[
		\omega = 2\pi \mathcal{D}_k f(\tau) \cdot F(\tau) \dd \tau + \left(-\log(v^{k/2} |f(\tau)|) \cdot \overline{\Delta_2 g(\tau)} - \frac{k}{2} \overline{g(\tau)} \right) \dd \overline{\tau}
	\]
	satisfies
	\[
		\dd \omega = \Delta_0 F(\tau) \cdot \log(v^{k/2} |f(\tau)|) \frac{\dd u \wedge \dd v}{v^2}.
	\]
	Applying Stokes' theorem then yields
	\[
		\langle \Delta_0 F, \log(v^{k/2} |f|)\rangle^\mathrm{reg} = \lim_{\varepsilon \to 0} \int_{\mathcal{F}(\Gamma_0(N))^\mathrm{reg}} \dd \omega = \lim_{\varepsilon \to 0} \int_{\partial \mathcal{F}(\Gamma_0(N))^\mathrm{reg}} \omega,
	\]
	where $\mathcal{F}(\Gamma_0(N))^{\mathrm{reg}}$ denotes a truncated fundamental domain for $\Gamma_0(N)$, obtained by removing $\varepsilon$-balls around the zeros and poles of $f$. The poles of $Df/f$ (arising from $\mathcal{D}_kf(\tau)$) in $\bbH$ contribute to $\mathcal{D}_F(\div(f))$, and the integrals around cusps yield the terms $C_\rho(f,g)$. (However, the limit as $\varepsilon \to 0$ that defines $C_\rho(f,g)$ does not necessarily exist in general).
\end{proof}

\subsection{Corollaries}

By applying \cref{thm:JKKM2025} to weight 2 polyharmonic Maass forms as the function $g(\tau)$, we can extend various results, including \cref{thm:BKO}, \cref{thm:valence}, \cref{thm:Rohrlich}, to the setting of $\Gamma_0(N)$, simultaneously. For instance, a partial extension of \cref{thm:BKO} to general level $N$ was given in~\cite{BKLOR2018}, but our theorem provides a full generalization of that result as shown in~\cite[Section 4.4]{JKKM2025}.

 We briefly review the construction of polyharmonic Maass forms using the Maass--Poincar\'{e} series, following the approach in~\cite{Matsusaka2020}. For $\Re(s) > 1$ and an integer $m$, the \emph{Maass--Poincar\'{e} series} $P_{N,k,m}(\tau,s)$ of weight $k \in 2\Z$ is defined by
\[
	P_{N,k,m}(\tau, s) \coloneqq \sum_{\gamma \in \Gamma_0(N)_\infty \backslash \Gamma_0(N)} \psi_{k,m}(v, s) \mathbf{e}(mu)|_k \gamma,
\]
where $\psi_{k,m}(\cdot, s) : \R_{>0} \to \C$ is defined by
\[
	\psi_{k,m}(v, s) = \begin{cases}
		\Gamma(2s)^{-1} (4\pi|m|v)^{-k/2} M_{\sgn(m)k/2, s-1/2}(4\pi |m|v) &\text{if } m \neq 0,\\
		v^{s-k/2} &\text{if } m=0,
	\end{cases}
\]
and $M_{\mu, \nu}(y)$ denotes the $M$-Whittaker function. It is known that $P_{N,k,m}(\tau, s)$ admits the meromorphic continuation to $s=1$, and its Laurent expansion around $s=1$,
\[
	P_{N,k,m}(\tau, s) = \sum_{r=-1}^\infty G_{N,k,m,r}(\tau) (s-1)^r
\]
yields a sequence of smooth modular forms $(G_{N,k,m,r}(\tau))_{r \ge -1}$ of weight $k$ on $\Gamma_0(N)$. As shown in~\cite[(3.1)]{Matsusaka2020}, we have
\begin{align}\label{eq:xi-P}
	\xi_k P_{N,k,m}(\tau, s) = \begin{cases}
		(4\pi|m|)^{1-k} (\overline{s}-k/2) P_{N,2-k,-m}(\tau, \overline{s}) &\text{if } m \neq 0,\\
		(\overline{s}-k/2) P_{N,2-k,0}(\tau, \overline{s}) &\text{if } m = 0.
	\end{cases}
\end{align}
By comparing the coefficients at $s=1$ on both sides, we obtain the relation satisfied by $G_{N,k,m,r}(\tau)$ under the action of $\xi_k$.

\begin{example}
	When $N = 1$ and $m=0$, the classical Kronecker limit formula shows that
	\[
		P_{1,0,0}(\tau, s) = \frac{3}{\pi} \frac{1}{s-1} - \frac{1}{2\pi} \log(v^6 |\Delta(\tau)|) + C + O(s-1),
	\]
	where $C$ is defined as in \cref{thm:Rohrlich}. From \eqref{eq:xi-P}, we find that
	\[
		G_{1,2,0,0}(\tau) = E_2^*(\tau) \coloneqq 1 - 24 \sum_{n=1}^\infty \sigma_1(n) q^n - \frac{3}{\pi v}
	\]
	satisfies $F(\tau) = \xi_2 E_2^*(\tau) = 3/\pi = G_{1,0,0,-1}(\tau)$. Therefore, we can apply \cref{thm:JKKM2025} to derive the valence formula (\cref{thm:valence}). Indeed, in this case, since $\Delta_0F(\tau) = 0$, we obtain
	\[
		\frac{3}{\pi} \sum_{[z] \in \SL_2(\Z) \backslash \bbH} \frac{\mathrm{ord}_z(f)}{|\PSL_2(\Z)_z|} = \mathcal{D}_F(\div(f)) = C_{i\infty}(f, E_2^*),
	\]
	and the constant can be explicitly calculated as 
	\[
		C_{i\infty}(f, E_2^*) = -\frac{3}{\pi} \mathrm{ord}_{i\infty}(f) + \frac{k}{4\pi}.
	\]
\end{example}

Similarly, by applying \cref{thm:JKKM2025} to $g(\tau) = G_{N,2,0,0}(\tau)$, $G_{N,2,0,1}(\tau)$, and $G_{N,2,m,1}(\tau)$ with $m>0$, we obtain higher-level analogues of \cref{thm:valence}, \cref{thm:Rohrlich}, and \cref{thm:BKO}, respectively. For more details and further generalizations, we refer the reader to our original article~\cite{JKKM2025}.

\bibliographystyle{amsalpha}
\bibliography{References} 

\end{document}